\documentclass[11pt, reqno]{amsart}
\usepackage{amsmath, amsfonts, amsthm, amssymb, multicol, mathtools, dsfont,verbatim}
\usepackage{graphicx}
\usepackage{float, hyperref}
\usepackage{scalerel,stackengine, subcaption}
\usepackage[usenames,dvipsnames,x11names]{xcolor}
\usepackage{enumitem}
\usepackage{pgfplots}
\usepgfplotslibrary{fillbetween}
\pgfplotsset{width=10cm,compat=1.9}
\usepackage[square,sort,comma,numbers]{natbib}
\makeatletter
\def\@setauthors{%
  \begingroup
  \def\thanks{\protect\thanks@warning}%
  \trivlist
  \centering\footnotesize \@topsep30\p@\relax
  \advance\@topsep by -\baselineskip
  \item\relax
  \author@andify\authors
  \def\\{\protect\linebreak}

  \normalsize\lowercase{\authors}%
  
	\ifx\@empty\contribs
  \else
    ,\penalty-3 \space \@setcontribs
    \@closetoccontribs
  \fi
  \endtrivlist
  \endgroup
}
\def\@settitle{\begin{center}
\LARGE\lowercase{\@title}
  \end{center}%
}
\makeatother
\newcommand{\authoremail}[1]{\email{\href{mailto:#1}{\color{lightblue}{#1}}}}
\newcommand{\authoraddress}[1]{\address{\normalfont{#1}}}

\usepackage{fancyhdr}
\numberwithin{equation}{section}

\newtheorem{thm}{Theorem}[section]
\newtheorem{lma}[thm]{Lemma}
\newtheorem{cor}[thm]{Corollary}

\newtheorem{prop}[thm]{Proposition}

\newtheorem{ques}[thm]{Question}

\DeclareMathOperator*\lowlim{\liminf}
\DeclareMathOperator*\uplim{\limsup}

\theoremstyle{definition}

\theoremstyle{remark}

\renewcommand{\epsilon}{\varepsilon}

\newcommand{\R}{\mathbb{R}}

\renewcommand{\geq}{\geqslant}
\renewcommand{\leq}{\leqslant}

\newcommand{\ubd}{\overline{\dim}_{\textup{B}}}
\newcommand{\lbd}{\underline{\dim}_{\textup{B}}}

\newcommand{\ad}{\dim_{\mathrm{A}} }
\newcommand{\qad}{\dim_{\mathrm{qA}} }
\newcommand{\as}{\dim^\theta_{\mathrm{A}} }

\newcommand{\hd}{\dim_{\mathrm{H}}  }
\newcommand{\bd}{\dim_{\mathrm{B}}  }

\newcommand{\rn}{\mathbb{R}^d}

\newcommand{\uid}{\overline{\dim}_{\,\theta}}
\newcommand{\lid}{\underline{\dim}_{\,\theta}}

\newcommand{\be}{\begin{equation}}
\newcommand{\ee}{\end{equation}}

\renewcommand{\epsilon}{\varepsilon}

\newcommand{\rd}{\mathbb{R}^d}

\newcommand{\fs}{\dim^\theta_{\mathrm{F}}}
\newcommand{\fd}{\dim_{\mathrm{F}}}
\newcommand{\sd}{\dim_{\mathrm{S}}}

\newcommand{\gdk}{G(d,k)}

\newcommand{\gto}{G(2,1)}

\newcommand{\gamdk}{\gamma_{d,k}}

\newcommand{\rk}{\mathbb{R}^k}

\newcommand{\N}{\mathbb{N}}
\renewcommand{\L}{\mathcal{L}}
\newcommand{\M}{\mathcal{M}}
\renewcommand{\S}{\mathcal{S}}

\newcommand{\spt}{\text{spt}\,}

\newcommand{\J}{\mathcal{J}}

\makeatletter
\DeclareRobustCommand\widecheck[1]{{\mathpalette\@widecheck{#1}}}
\def\@widecheck#1#2{%
    \setbox\z@\hbox{\m@th$#1#2$}%
    \setbox\tw@\hbox{\m@th$#1%
       \widehat{%
          \vrule\@width\z@\@height\ht\z@
          \vrule\@height\z@\@width\wd\z@}$}%
    \dp\tw@-\ht\z@
    \@tempdima\ht\z@ \advance\@tempdima2\ht\tw@ \divide\@tempdima\thr@@
    \setbox\tw@\hbox{%
       \raise\@tempdima\hbox{\scalebox{1}[-1]{\lower\@tempdima\box
\tw@}}}%
    {\ooalign{\box\tw@ \cr \box\z@}}}
\makeatother

\stackMath
\newcommand\reallywidehat[1]{%
\savestack{\tmpbox}{\stretchto{%
  \scaleto{%
    \scalerel*[\widthof{\ensuremath{#1}}]{\kern.1pt\mathchar"0362\kern.1pt}%
    {\rule{0ex}{\textheight}}
  }{\textheight}%
}{2.4ex}}%
\stackon[-6.9pt]{#1}{\tmpbox}%
}
\definecolor{lightblue}{HTML}{2B77A4}
\colorlet{plotblue}{LightSkyBlue3!80}
\definecolor{darkred}{HTML}{9E0D0D}
\definecolor{darkyellow}{HTML}{b3b300}
\definecolor{darkorange}{HTML}{D86129}
\hypersetup{
	colorlinks=true,
	linkcolor=darkred,
	urlcolor=darkred,
	citecolor=lightblue
}
\title{Applications of dimension interpolation to orthogonal projections}

\author{Jonathan M. Fraser}
\authoraddress{J. M. Fraser, University of St Andrews, Scotland}
\authoremail{jmf32@st-andrews.ac.uk}
\thanks{JMF was financially supported by  a  \emph{Leverhulme Trust Research Project Grant} (RPG-2023-281) and an \emph{EPSRC Standard Grant} (EP/Y029550/1).}

\begin{document}
\maketitle
\thispagestyle{empty}

\begin{abstract}
 Dimension interpolation is a novel programme of research which attempts to unify the study of fractal dimension by considering various spectra which live in between well-studied notions of dimension such as Hausdorff, box, Assouad and Fourier dimension.  These  spectra often reveal novel features not witnessed by the individual notions and this information has applications in many directions.  In this survey article, we discuss dimension interpolation broadly and then focus on applications to the dimension theory of orthogonal projections. We focus on three distinct applications coming from three different dimension spectra, namely, the Fourier spectrum, the intermediate dimensions, and the Assouad spectrum.  The celebrated Marstrand--Mattila projection theorem gives the Hausdorff dimension of the orthogonal projection of a Borel set in Euclidean space for almost all orthogonal projections.  This result has inspired much further research on the dimension theory of projections including the consideration of dimensions other than the Hausdorff dimension, and the study of the exceptional set in the Marstrand--Mattila   theorem.  \\ \\
  \emph{Mathematics Subject Classification}: primary: 28A80, 42B10; secondary: 28A75, 28A78.
\\
\emph{Key words and phrases}: dimension interpolation, projections, Marstrand's  Theorem, Hausdorff dimension, box dimension, Assouad dimension, Fourier dimension, Assouad spectrum, intermediate dimensions,  Fourier spectrum.
\end{abstract}

\tableofcontents

\section{Dimension theory} \label{intro}

One of the most basic questions one can ask about a geometric object is: how big is it? Indeed, children are quickly taught about the length of a line, the area of a circle, and the volume of a sphere when they embark upon their mathematical journeys. However, in highlighting these simple examples, we have already skipped over an even more fundamental question.  Before we measure the length of a line,  the area of a circle, and the volume of a sphere, we must first identify  that these objects are respectively 1, 2, and 3 dimensional and that it is therefore  appropriate to measure their size using 1-dimensional length, 2-dimensional area, and 3-dimensional volume. This is also the case in more advanced geometry, where asking for the dimension of a (potentially fractal) set is one of the most fundamental questions.  However, a perhaps even more fundamental question is to ask how to \emph{define} `dimension' in a robust and useful way.  It turns out that there are many ways to do this, each notion picking up a different feature of the underlying object.  Understanding the subtle relationships and differences between distinct notions of dimension is one of the joys of  fractal geometry and dimension theory.  This survey article will focus on four distinct notions of fractal dimension: the Hausdorff, box, Assouad and Fourier dimensions.  In this section we briefly define all four and give some of their basic connections.   For more background on box and Hausdorff dimensions, see \cite{Fal03}; for more on the Assouad dimension, see \cite{jon:book, mackaytyson,robinson}; and for more on the Fourier dimension, see \cite{Mat15}.

 For a non-empty bounded  set $X\subseteq \rn$ and a scale $r>0$,  let $N_r(X)$ be the minimum number of sets of diameter $r$ that can cover $X$.  Thus $N_r(X)$ is a simple measure of how large $X$ is at scale $r$.  The {\em lower} and {\em upper box dimensions}  of $X$ are defined by
\[
\lbd X\ =\ \varliminf_{r\to 0} \frac{\log  N_r(X)}{-\log r}
\quad \mbox{ and }\quad  \ubd X\ = \ \varlimsup_{r\to 0} \frac{\log  N_r(X)}{-\log r}
\]
and these capture the growth rate of $N_r(X)$ as the scale $r$ shrinks.  If $\lbd X = \ubd X$ then we write $\bd X$ for the common value and call it simply the \emph{box dimension} of $X$.  It is sometimes useful to note that  $N_r(X)$ can be replaced with other related ways of measuring the coarse size of $X$ at scale $r$, such as the maximal size of an $r$-separated subset etc, without changing the dimensions; see \cite{Fal03}.   A more sophisticated notion, which is similar in spirit, is the Hausdorff dimension.  Here and throughout $|X|$ denotes the diameter of a set $X \subseteq  \rd$.  The \emph{Hausdorff dimension} of $X$ can be defined by
\begin{eqnarray*}
\hd  X\ = \  \inf \bigg\{ \  \alpha>0 &:&  \text{for all $\varepsilon>0$ there exists a cover $\{U_i\}$ of $X$} \\
&\,& \hspace{30mm}  \text{ such that  $\sum_i |U_i|^\alpha < \varepsilon$ } \bigg\}.
\end{eqnarray*}
The key difference here is that sets with vastly different diameters  are permitted in the cover and their contribution to the `dimension' is weighted according to their diameter. 

 Both the Hausdorff and box dimensions measure the size of the whole set, giving rise to an `average global dimension'. It is often the case that more extremal or local information is required, for example in embedding theory, see \cite{robinson}.  The \emph{Assouad dimension} is designed to capture this information and is defined by
\begin{eqnarray*}
\dim_\text{A} X \ = \  \inf \bigg\{ \  \alpha>0 &:& \text{     there exists a constant $C >0$ such that,} \\
&\,& \hspace{10mm}  \text{for all $0<r<R $ and $x \in X$, } \\ 
&\,&\hspace{20mm}  \text{$ N_r\big( B(x,R) \cap X \big) \ \leq \ C \bigg(\frac{R}{r}\bigg)^\alpha$ } \bigg\}.
\end{eqnarray*}
The key difference between this and the box dimension, say,  is that one does not seek covers of the whole set, but only a small ball, and  the expected covering number  is appropriately normalised.     It is a simple and instructive exercise to show  that
\[
0 \leq \hd X \leq  \lbd X \leq \ubd X \leq \ad X \leq d
\]
for any bounded $X \subseteq \mathbb{R}^d$.  Moreover,  these inequalities can be strict inequalities or equalities in any combination. Sets for which the Assouad, box, and Hausdorff dimensions coincide must exhibit some form of homogeneity, both in space and scale.   For example, if $X$ is Ahlfors-David regular then $\hd X = \bd X = \ad X$.

The Fourier dimension is of a different flavour and is not defined in terms of covers.  However, by connecting the Hausdorff dimension to energies and Fourier transforms, the Fourier dimension quickly adopts it place in this story.   Frostman's lemma allows us to write the Hausdorff dimension of a Borel set $X$ in terms of the $s$-energy $I_{s}$ of measures $\mu\in\M(X)$, where $\M(X)$ is the set of  positive and finite Borel measures supported on $X$.  The  $s$-energy of $\mu$ is
\begin{equation*}
		I_{s}(\mu) = \iint |x-y|^{-s}\,d\mu(x)\,d\mu(y),
\end{equation*}
and   the Hausdorff dimension of a Borel set $X$ may be defined as
\begin{equation*}
		\hd X = \sup\{ s\geq0 : \exists \mu\in\M(X) : I_{s}(\mu)<\infty \}.
\end{equation*}
Deriving the Hausdorff dimension via this alternative definition is often referred to as the potential theoretic method;  see \cite{Fal03,Mat15} for more details.  These energy integrals can be expressed in terms of the Fourier transform of the measure and this connection opens up a rich interplay between dimension theory and Fourier analysis. If $\mu$ is a finite Borel measure, we define its \emph{Fourier transform} by
\begin{equation*}
		\widehat{\mu}(z) = \int e^{-2\pi iz\cdot x}\,d\mu(x).
\end{equation*}
Using that the Fourier transform of the Riesz kernel $|z|^{-s}$ is, in the distributional sense, a constant multiple of $|z|^{s-d}$, and Parseval's theorem, the $s$-energy of $\mu\in\M(\rd)$, for $s\in(0,d)$, may be expressed as
\begin{equation*}
		I_{s}(\mu) =C({d,s}) \int \big| \widehat{\mu}(z) \big|^2 |z|^{s-d}\,dz
\end{equation*}
where $C(d,s)$ is a constant depending only on $s$ and $d$. This relationship between the Hausdorff dimension of sets and the Fourier transform of measures they support motivates the definition of the \emph{Fourier dimension} of a finite Borel measure as
\begin{equation*}
		\fd \mu = \sup\big\{ s\geq0 : \sup_z\big| \widehat{\mu}(z) \big|^2  |z|^{s} < \infty \big\},
\end{equation*}
and of a Borel set  $X\subseteq\rd$ as
\begin{equation*}
		\fd X = \sup\big\{ \min\{ \fd \mu,  d\} : \mu\in\M(X) \big\}.
\end{equation*}
For a Borel set $X\subseteq\rd$, 
\[
0\leq\fd X\leq\hd X
\]
and   these inequalities can be strict in any combination. Sets for which the Fourier and Hausdorff dimensions coincide are called \emph{Salem sets}. Constructing non-trivial deterministic Salem sets is challenging, but random examples abound.  We refer the reader to \cite{FH23, Kau81, Ham17, Mat15} for a more detailed   history of Salem sets.  

We end this introductory section with a simple example where the four notions of dimension discussed here take on distinct values.  Let $X \subseteq \mathbb{R}^2$ be given by
\[
X =\{1/n\}_{n \in \mathbb{N}} \times [0,1].
\]
Then it is straightforward to show that
\[
\fd X = 0; \quad \hd X = 1; \quad \bd X = 3/2; \quad  \ad X = 2;
\]
and the unfamiliar reader may enjoy demonstrating these claims for themself.  Each notion of dimension is describing the size of $X$ in a different way and they come to very different conclusions!

\section{Dimension interpolation}

Dimension interpolation is a novel programme in the dimension theory of fractals and outlines a new perspective in the way one considers dimension.   Historically, the different notions of dimension described in the previous section have been  considered largely  in isolation.  However, dimension interpolation suggests that we should  try to view them as different facets of the same object. This approach will  give rise to a continuum of dimensions, which more fully describe the geometric  structure of the space.  Moreover, this will yield a more nuanced understanding of the individual notions of dimensions as well as   better applications to a variety of problems; see \cite{Fra19} for an introductory survey.

  More concretely, given `dimensions' $\dim$ and $\textup{Dim}$ which generally satisfy $\dim X \leq \textup{Dim} \, X$, we wish to understand the (unexplored) gap between the dimensions by introducing a continuum of dimension-like functions $\{f_\theta\}_{\theta \in [0,1]}$  which (ideally) satisfy, for all reasonable $X$:  
\begin{itemize}
\item $f_0(X) = \dim X$  
\item $f_1(X) = \textup{Dim} \, X$  
\item $\dim X \leq f_\theta(X) \leq \textup{Dim} \, X, $ for $ \theta \in (0,1)$   
\item $f_\theta(X)$ is continuous in $\theta$, certainly for $\theta \in (0,1)$ and ideally at the `endpoints'
\item the definition of $f_\theta $ is `natural'  and incorporates aspects of both $\dim$ and $\textup{Dim}$
\item $f_\theta$ gives rise to a `rich theory'.
\end{itemize}
The most important of these points are the final two. It is vital that the function $\theta \mapsto f_\theta(X)$ is ripe with readily interpreted, meaningful, and nuanced information concerning the geometry of $X$.    If such an interpolation   $f_\theta(X)$ can be defined,  then the benefits are likely to include:
\begin{itemize}
\item a better understanding of  $\dim$ and $\textup{Dim}$ 
\item a more unified and comprehensive theory of  dimension
\item an explanation or exhibition  of one type of behaviour changing into another   
\item more information, leading to better applications in many directions
\item a large new set of questions and research directions.
\end{itemize}
In the following subsections we describe three concrete examples of this philosophy in action: the Fourier spectrum (which interpolates between the Fourier and Hausdorff dimensions);  the intermediate dimensions (which interpolate between the Hausdorff  and box dimensions); and  the Assouad spectrum (which interpolates between the box and Assouad dimensions).

The word `interpolation' can mean many things, but it often refers to constructing a function with some desired properties based on knowledge of the outputs at particular values. In dimension interpolation we are not interpolating based on knowledge of $\dim X$ and $\textup{Dim} \, X$, but rather based on knowledge of $\dim $ and $\textup{Dim}  $, that is, knowledge of the \emph{concepts} underpinning the dimensions, not the specific values in a particular instance.

\subsection{The Assouad  spectrum}

The Assouad  spectrum was  introduced by   Fraser and Yu in \cite{assouadspectrum} to interpolate between (upper) box  dimension and Assouad dimension. This was the first appearance of `dimension interpolation'.  The first observation is that the definition of  Assouad dimension uses  two scales (the `localisation scale' $R$ and the `covering scale' $r$) but the box dimension just uses one scale.  If the covering scale is very much smaller than the localisation scale, then one might expect the upper box dimension to appear (heuristically, covering the small piece is not so different from covering the whole set).  Moreover, pairs of scales which witness the Assouad dimension are often closer together than one might expect (with the obvious caveat that $R/r \to \infty$).  The Assouad spectrum is defined by using the interpolation parameter $\theta \in (0,1)$  to fix the relationship between these two scales.  It turns out that the most interesting way to do this is to set $r= R^{1/\theta}$ for $R \leq 1$. 

More precisely, the   \emph{Assouad spectrum} of $X \subseteq \rd$ at $\theta \in (0,1)$ is defined by
\begin{align*}
\as X  \ = \    \inf \bigg\{ \  \alpha &: \text{   there exists  a constant  $C >0$ such that,} \\
 &  \text{for all $0<  R<1$ and $x \in X$, }
  \text{$ N_{R^{1/\theta}} \big( B(x,R) \cap X \big) \ \leq \ C \bigg(\frac{R}{R^{1/\theta}}\bigg)^\alpha$ } \bigg\}.
\end{align*}
 One may think of the Assouad spectrum as the function $\theta \mapsto \as X$.  The related   \emph{upper Assouad spectrum} (or regularised spectrum) of $X \subseteq \rd$ at $\theta \in (0,1)$ is defined by
\begin{align*}
\overline{\dim}_\textup{A}^\theta X  \ = \    \inf \bigg\{ \  \alpha &: \text{   there exists  a constant  $C >0$ such that,} \\
 &  \text{for all $0<r \leq R^{1/\theta} < R<1$ and $x \in X$, }
  \text{$ N_{r} \big( B(x,R) \cap X \big) \ \leq \ C \bigg(\frac{R}{r}\bigg)^\alpha$ } \bigg\}.
\end{align*}
  Sometimes, as we shall see in this survey, it is useful to talk about  the upper Assouad spectrum in certain applications.  That said, it was proved in \cite{canadian} that
\[
\overline{\dim}_\textup{A}^\theta X = \sup_{\theta' \in (0,\theta)} \dim_\textup{A}^{\theta'}  X
\]
and so the upper Assouad spectrum contains less information that the Assouad spectrum itself.  The upper Assouad spectrum is clearly non-decreasing in $\theta$ but the Assouad spectrum need not be.  However, in most commonly studied situations it is non-decreasing and therefore the two spectra coincide.  Finally, the \emph{quasi-Assouad dimension} can be  defined by
\[
\qad X = \lim_{\theta \nearrow 1} \as X.
\]
The quasi-Assouad dimension was introduced by L\"u and Xi in \cite{quasiassouad} but the above formulation  in terms of the Assouad spectrum  was proved in \cite{canadian}. Generally,  for $\theta \in (0,1)$,
\[
  \ubd X \leq  \as X \leq \overline{\dim}_\textup{A}^\theta X  \leq \qad X \leq \ad X
\]
and
\[
\as X \leq \min\bigg\{\frac{\ubd X}{1-\theta} , \ \qad X\bigg\}.
\]
In particular, $\as X \to \ubd X$ as $\theta \to 0$, which one may think of as continuity at the endpoint $\theta=0$.  The possible functions which can be realised as the Assouad spectrum of some set were completely classified in \cite{specclass}.  It turns out that the family of possible Assouad spectra is very rich indeed.

The Assouad spectrum has already found numerous applications, some coming in surprising areas.  These applications include $L^p \to L^q$ mapping properties of spherical maximal functions \cite{anderson, roos, beltran}, weak embeddability problems  \cite{stathis}, the spiral winding problem \cite{spirals}, certain H\"older regularity problems \cite{holder},    quasiconformal mapping problems \cite{stathistyson,stathisquasi}, and to the Sullivan dictionary from conformal dynamics \cite{bullams}.   Applications to the dimension theory of orthogonal projections were provided in \cite{ffs} and we exhibit some of these applications here in  Section \ref{asssec}.

\subsection{The intermediate dimensions}

Intermediate dimensions were introduced by Falconer, Fraser and Kempton in \cite{intdims} to interpolate between the Hausdorff   and box dimensions. The first observation is that the definitions of Hausdorff and box dimension are both in terms of covers.  The difference is that for Hausdorff dimension there is no restriction on the relative sizes of  diameters of sets used in the cover, but for the box dimension all covering sets must have the same diameter (well, they may as well have the same diameter).  If the Hausdorff and box dimensions are distinct, then the set exhibits inhomogeneity in scale witnessed by the fact that one may find subtle efficient covers using sets with dramatically different sizes.  The intermediate dimensions are defined by using the interpolation parameter $\theta \in (0,1)$ to restrict the range of allowable diameters used in the covers.  It turns out that the most interesting way to do this is to insist that sets $U,V$ used in a cover must satisfy $|U| \leq |V|^\theta$.

More precisely,  for $X \subset \R^d$ and $0 < \theta \leq 1$, the  {\em lower intermediate dimension} of $X$ is defined by
\begin{align*}
\lid X=  \inf \big\{& s\geq 0  :  \mbox{ \rm for all $\epsilon >0$ and all $r_0>0$,  there exists $0<r\leq r_0$} \\
 & \mbox{ \rm and a cover $ \{U_i\} $ of $X$ such that  $r^{1/\theta} \leq  |U_i| \leq r $ and 
 $\sum |U_i|^s \leq \epsilon$}  \big\}\nonumber
\end{align*}
and the corresponding {\em upper intermediate dimension} by
\begin{align*}
\uid X =  \inf \big\{& s\geq 0  :  \mbox{ \rm for all $\epsilon >0$ there exists $r_0>0$ such that for all $0<r\leq r_0$,} \\
 & \mbox{ \rm there is a cover $ \{U_i\} $ of $X$ such that  $r^{1/\theta} \leq  |U_i| \leq r$ and 
 $\sum |U_i|^s \leq \epsilon$}  \big\}.\nonumber
\end{align*}
We can include  $\theta = 0$ in this definition by adopting the convention that $r^{1/0}=0$ for $r \in (0,1)$.  In that case, there are no lower bounds  on the diameters of covering sets and we recover the Hausdorff dimension in both the lower and upper case. When $\theta=1$ all covering sets are forced to have the same diameter and we recover the lower and upper box dimensions, respectively.

Various properties of intermediate dimensions are established in \cite{intdims}; see also \cite{banajigen}. In particular, $\lid X$ and $\uid X$ are monotonically increasing in $\theta\in [0,1]$, are continuous except perhaps at $\theta = 0$, and are invariant under bi-Lipschitz mappings. The possible functions which can be realised as the intermediate dimensions of some set were completely classified in \cite{intclass}.  Once again, it turns out that there is a rich family of possibilities.

The intermediate dimensions have already found numerous applications.  These applications include the dimension theory of Brownian images \cite{burrell}, bi-Lipschitz classification problems \cite{banajicarpets},  multifractal analysis \cite{banajicarpets}, and Sobolev mapping problems \cite{frasertyson}.   Applications to the dimension theory of orthogonal projections were provided in \cite{bff21} and we exhibit some of these applications here in Section \ref{intsec}.

\subsection{The Fourier spectrum}

The Fourier spectrum was introduced by   Fraser  in \cite{fourierspectrum} to interpolate between the Fourier   and Hausdorff dimensions. The first observation is that the definitions of Fourier  and Hausdorff dimension can both be expressed in terms of decay of the Fourier transform of measures supported on the set. The difference is that for Fourier dimension the decay is measured in $L^\infty$, but for the Hausdorff dimension the decay is measured in an averaged $L^2$ sense.   If the Fourier  and Hausdorff dimensions are distinct, then one can find measures on the set whose Fourier transform decays on average but it is not possible to achieve the same decay rate uniformly.  The Fourier spectrum is defined by using the interpolation parameter $\theta \in (0,1)$ to consider different averages, namely in an appropriate $L^p$ sense, where $p=2/\theta$.

The definition of Hausdorff dimension for sets using energy integrals can be extended naturally to measures.  We define  the \emph{Sobolev dimension} of a measure $\mu\in\M(\rd)$ by
\begin{equation*}
	\sd \mu = \sup\bigg\{ s \in \R : \int_{\rd}\big|\widehat{\mu}(x)\big|^2 |x|^{s-d}\,dx<\infty \bigg\}.
\end{equation*}
This concept goes back to Peres--Schlag \cite{PS00}; see also \cite{Mat15}.  For any Borel measure $\mu \in \M(\rd)$, $0 \leq \fd \mu\leq\sd \mu$ and $\hd \mu \geq \min\{d, \sd \mu\}$, where $\hd\mu$ refers to the Hausdorff dimension of a measure $\hd\mu = \inf\{ \hd X : X\text{ is a Borel set with }\mu(X)>0 \}$. Contrary to what one might expect, both the Fourier and Sobolev dimensions of measures may exceed the Hausdorff dimension of the ambient space. Take as an example the Lebesgue measure restricted to $[0,1]$, $\L^1\big|_{[0,1]}$. This measure satisfies $\fd \L^1\big|_{[0,1]} = \sd\L^1\big|_{[0,1]} = 2>\hd[0,1]$.

In order to define the Fourier spectrum, we first define $(s,\theta)$-energies of a measure $\mu\in\M(\rd)$, for $\theta\in(0,1]$ and $s\geq0$, by
\begin{equation*}
	\J_{s,\theta}(\mu) = \bigg( \int_{\rd} \big| \widehat{\mu}(z) \big|^{\frac{2}{\theta}}|z|^{\frac{s}{\theta}-d}\,dz \bigg)^\theta,
\end{equation*}
and, for $\theta = 0$, by
\begin{equation*}
	\J_{s,0}(\mu) = \sup_{z\in\rd} \big| \widehat{\mu}(z) \big|^2|z|^{s}.
\end{equation*}
Then the \emph{Fourier spectrum} of $\mu$ at $\theta$ is
\begin{equation*}
	\fs \mu = \sup\{ s \in \R : \J_{s,\theta}(\mu)<\infty \},
\end{equation*}
and for each $\theta\in[0,1]$, $\fd \mu\leq\fs \mu\leq\sd \mu$, with equality on the left if $\theta= 0$ and equality on the right if $\theta = 1$. As a function of $\theta$, $\fs\mu$ is concave and continuous for $\theta\in(0,1]$ by \cite[Theorem~1.1]{fourierspectrum} and, in addition,  continuous at $\theta=0$ provided $\mu$ is compactly supported by \cite[Theorem~1.3]{fourierspectrum}. Further, it was proved in \cite[Proposition 4.2]{CFdO24} that for compactly supported $\mu$, $\fs \mu \leq \fd \mu + d \theta$ for  all $\theta \in [0,1]$.

For  a Borel set $X\subseteq\rd$, the \emph{Fourier spectrum} is defined by
\begin{equation*}
	\fs  X = \sup\big\{ \min\{\fs \mu, d\} : \mu\in\M(X) \big\}.
\end{equation*}
Then,  for all $\theta\in[0,1]$, $\fd X\leq\fs  X \leq \hd X$, with equality on the left if $\theta= 0$ and equality on the right if $\theta = 1$. Moreover,  $\fs X$ is continuous for all $\theta\in[0,1]$ by \cite[Theorem~1.5]{fourierspectrum}.  Unlike for measures, $\fs X$ need not be concave, but it does satisfy $\fs X \leq \fd X + d \theta$ for  all $\theta \in [0,1]$.

The Fourier spectrum has already found numerous applications where one uses the additional information provided by the spectrum of dimensions to obtain stronger results than one can get by appealing to the Fourier dimension alone.  These applications include new Hausdorff dimension estimates in the Falconer distance problem \cite[Section 7]{fourierspectrum}, sumset type problems \cite[Section 6]{fourierspectrum}, and the celebrated restriction problem in harmonic analysis \cite{restriction}.  Applications to the dimension theory of orthogonal projections were provided in \cite{ana} and we exhibit some of these applications here in Section \ref{specsec}.

\subsection{Simple example: revisited}

Recall the simple example described above used to distinguish between our four notions of fractal dimension. That is,  $X \subseteq \mathbb{R}^2$   given by
\[
X =\{1/n\}_{n \in \mathbb{N}} \times [0,1].
\]
Of course, the reader is now wondering how dimension interpolation handles this example.  One can show, with a little more work this time, that for all $\theta \in (0,1)$
\[
\fs X = \theta; \quad \dim_\theta X = \frac{1+2 \theta}{1+\theta}; \quad \as X = \min\left\{ \frac{3/2-\theta}{1-\theta},  2\right\}
\]
and the  reader may enjoy demonstrating these claims for themself.  One might look at this example and declare that the Hausdorff dimension is clearly 1 and be done with it.  But observe that---even in this very simple case---dimension interpolation has unearthed much more geometric information about $X$.  Moreover, we uncover a complete interpolation between 0 and 2, which observes  the four isolated notions of fractal dimension in the process.  We are also led to many further questions pertaining to finer geometric features of $X$: why is the Fourier spectrum affine (and as small as it can be given the Fourier and Hausdorff dimensions)?  why are the intermediate dimensions strictly concave?  why does the Assouad spectrum have a phase transition at $\theta=1/2$? why does the Assouad spectrum reach the Assouad dimension for values of $\theta<1$?  Are these behaviours typical for other examples? etc etc.

\begin{figure}[H] 
	\centering
	\includegraphics[width=\textwidth]{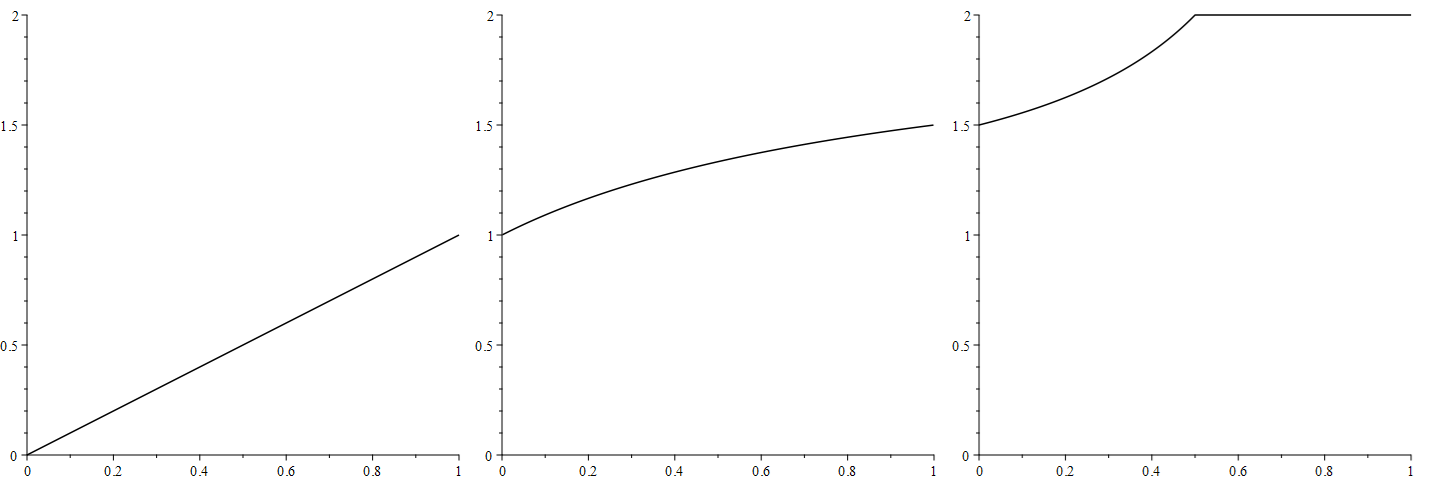}
\caption{\emph{Complete interpolation}: plots of the Fourier spectrum (left), the intermediate dimensions (centre), and the Assouad spectrum (right) as functions of $\theta \in (0,1)$ for the simple example $X =\{1/n\}_{n \in \mathbb{N}} \times [0,1]$. }
\end{figure}

It turns out that many different properties of the interpolation functions can be observed in more complicated cases. There are many results in this direction, but we briefly mention a few striking examples.  For self-affine Bedford--McMullen carpets the intermediate dimensions typically have infinitely many phase transitions and points with infinite right derivative \cite{banajicarpets}.  For self-affine Gatzouras--Lalley carpets the Assouad spectrum can be differentiable on $(0,1)$ including several strictly concave parts and several strictly convex parts \cite{lgspectrum}.  For the arclength measure on the moment curve in $\rd$ the Fourier spectrum is piecewise affine with $d-2$ phase transitions \cite{restriction},  and for certain Riesz products the Fourier spectrum is strictly concave and smooth \cite{fourierspectrum}.

\section{Dimension theory of orthogonal projections}

For integers $1\leq k<d$, we write $G(d,k)$ for the Grassmannian manifold of $k$-dimensional subspaces $V\subseteq\rd$.  This space may be equipped with the invariant Borel probability measure $\gamdk$ obtained from the Haar measure on the topological group of rotations around the origin. We write  $P_{V}:\rd\to V$ for the orthogonal projection onto $V\in\gdk$, which we identify with $\rk$.   

A central question in fractal geometry and geometric measure theory is to understand how geometric properties of a set $X \subseteq \rd$ relate to the geometric properties of the projections $P_V(X)$.  For example one may consider how the fractal dimensions of $X$ are related to the fractal dimensions of $P_V(X)$.  It does not take long to observe that this is a non-trivial problem and that we should not try to describe all the projections simultaneously.    For example, even for the simplest sets there may be particular projections which behave badly in the sense that the dimension seems not to reflect the dimension of $X$.  For example consider a line segment in the plane.  This has Hausdorff dimension 1 and all projections onto 1-dimensional subspaces will have Hausdorff dimension 1 (as expected?) apart from the (unique) 1-dimensional subspace which is orthogonal to the line segment.  On the other hand, since Hausdorff (and box) dimension cannot increase under Lipschitz maps, we at least get
\[
\hd P_V(X) \leq \min\{k, \hd X\}
\]
for \textbf{all} $V \in G(d,k)$. The Assouad dimension may increase under Lipschitz maps and so the corresponding bound does not hold for Assouad dimension---more on this later. Given the above discussion, one may aim to describe the dimensions of $P_V(X)$ for generic $V$.  For example, for almost all $V$ with respect to $\gamdk$  or for all $V$ outside of an exceptional set with small Hausdorff dimension.  The most important result in this direction is the seminal Marstrand--Mattila projection theorem, which we describe in the next section.

\subsection{Hausdorff dimension and the Marstrand--Mattila theorem}

One of the most well-known and influential results in fractal geometry is Marstrand's projection theorem.  This was proved in the plane by Marstrand   \cite{Mar54} and a simpler potential theoretic proof was given by Kaufman \cite{Kau68}.  Kaufman's proof was later generalised to higher dimensions by Mattila     \cite{Mat75} and the result is that  for Borel sets $X\subseteq\rd$ and $\gamdk$ almost all $V\in\gdk$, 
\[
\hd P_{V}(X) = \min\{k, \hd X \}.
\]
This result stimulated a huge amount of activity in fractal geometry, geometric measure theory and related fields.  We do not attempt a comprehensive survey here, but refer the reader to the (not that recent) surveys \cite{FFJ15, Mat14} and \cite{Shm15}, the latter of which focuses on dynamical situations, as well as the more recent survey \cite{falconer25}.  One direction which is relevant for us is the study of the exceptional set in the Marstrand--Mattila theorem.  The exceptional set has zero $\gamdk$ measure, but it may still be large in the sense of Hausdorff dimension.  For example, Peres and Schlag \cite[Proposition 6.1]{PS00} gave the following upper bound for Borel sets $X\subseteq\rd$. For $u\in[0,k]$,
\begin{equation}\label{eq:peresschlagbound}
    \hd\{ V\in \gdk : \hd P_{V}(X)< u \}\leq k(d-k)+u-\hd X.
\end{equation}
This bound is not sharp in general, but the situation is better understood in the plane.  Oberlin \cite{Obe12} conjectured that if $ \frac{\hd X}{2} \leq u \leq \min\{\hd X, 1\}$, then
\begin{equation}\label{eq:oberlin}
		\hd \{ V\in\gto : \hd P_{V}(X)<u \}\leq 2u-\hd X
\end{equation}
and it is well-known that one cannot do better than this.  Oberlin's conjecture \eqref{eq:oberlin}  was  recently proved by Ren and Wang  \cite[Theorem~1.2]{RW23} in a breakthrough which  built on significant recent progress  from various people; see, for example,  \cite{bourgain, guth,OS23,OS23+,OSW24} and the references therein.

\subsection{Dimension profiles and box dimension} \label{boxpotential}

Given Marstrand's projection theorem, one of the most natural follow up questions is whether or not there is an analogue for other notions of dimension.  It turns out that there are natural analogues for the box dimensions; that is, the box dimensions of orthogonal projections take on a constant value almost surely.  However, the constant value is rather more complicated than simply $\min\{k,\dim X\}$ and in general is described by a `dimension profile'.  These results were first established by Falconer and Howroyd \cite{falconerhowroyd,falconerhowroyd2,howroyd}; see also \cite{falconerprofile, falconerprofile2}.

Dimension profiles may be defined in terms of capacities with respect to certain kernels \cite{falconerprofile}. For $s \in [0,d]$ and $r>0$ we define the  kernel
\be 
 \phi_r^s(x)= \min\bigg\{ 1, \bigg(\frac{r}{|x|}\bigg)^s\bigg\} \quad (x\in \rn).\label{ker}
\ee
For a  non-empty compact  $X \subseteq \rn$,  the \emph{capacity}, $C_r^s(X)$, of $X$ with respect to this kernel is  given by
\[
\frac{1}{C_r^s(X)}\  = \ \inf_{\mu \in {\mathcal M}(X)}\int\int \phi_r^s(x-y)d\mu(x)d\mu(y)
\]
where  ${\mathcal M}(X)$ denotes the collection of Borel probability measures supported by $X$.  The double integral inside the infimum is called the {\it energy} of $\mu$ with respect to the kernel. Compare this with the $s$-energy $I_s(\mu)$ and the potential theoretic approach to Hausdorff dimension where a different kernel is used.  The capacity of a general bounded set is taken to be that of its closure.
For bounded $X \subseteq \rn$ and $s>0$ we define the {\it lower} and {\it upper box dimension   profiles} of $X$ by
\[
\underline{\mbox{\rm dim}}_{\rm B}^s X \ =\   \varliminf_{r\to 0} \frac{\log  C_r^s(X)}{-\log r},\quad \overline{\mbox{\rm dim}}_{\rm B}^s X \ =\   \varlimsup_{r\to 0} \frac{\log  C_r^s(X)}{-\log r}.
\]
In particular, by \cite[Corollary 2.5]{falconerprofile} if $s\geq d$ then
\[
\underline{\mbox{\rm dim}}_{\rm B}^s X \ =\  \lbd X, \ \  \overline{\mbox{\rm dim}}_{\rm B}^s X \ =\  \ubd X,
\]
but for  $s < d$ the dimension profiles give the almost sure  dimensions of projections of sets as well as information on the size of the set of exceptional projections, as follows.

\begin{thm}{\rm \cite[Theorems 1.1, 1.2]{falconerprofile}}\label{mainA}

\noindent $(i)$ Let $1\leq k<d$ be an integer.  For almost all $V \in G(d,k)$, if $X\subseteq \rn$ is bounded
\[
\lbd P_V (X) = \lbd^k X, \ \ \text{and} \ \ \ubd P_V (X) = \ubd^k X.
\]
\noindent $(ii)$ For $0<s<k$, if $X\subseteq \rn$ is bounded
\[
\hd \{ V \in G(d,k) : \lbd P_V (X) < \lbd^s X\} \leq k(d-k)-(k-s),
\]
\[
\hd \{ V \in G(d,k) : \ubd P_V (X) < \ubd^s X\} \leq k(d-k)-(k-s).
\]
\end{thm}

Although the values of $\lbd P_V (X)$ and $ \ubd P_V (X)$ are constant for almost all $V\in G(d,k)$,  this constant can take any value in the range
 \begin{equation} \label{boxdimsbounds}
\frac{\ubd X}{1+(1/k-1/d)\ubd X} \  \leq \ \ubd^k X \leq \min \{k, \ubd X\},
\end{equation}
with analogous inequalities for lower box dimension. These inequalities were established in \cite{falconerhowroyd, falconerhowroyd2, howroyd} using dimension profiles directly, with examples showing them to be best possible in \cite{falconerhowroyd}.  A simpler approach using capacities was given recently in \cite{falconerprofile}. 

\subsection{Further Marstrand--type theorems?}

It turns out that there is no Marstrand theorem for Assouad dimension, that is, the Assouad dimension of the orthogonal projection of a compact set in $\rd$ does not necessarily take on a constant value almost surely. This was established by Fraser and Orponen \cite{FO17} and further strange behaviour was exhibited in \cite{antti}.  For example,  there  exist examples where $\ad P_V(X)$ takes a different value for each distinct $V$! With this in mind, it is natural to ask about the Assouad spectrum, but this is open.

\begin{ques} \label{ques1}
  Is there a Marstrand theorem for the Assouad spectrum?  More precisely, fix a compact  set $X$  in $\rd$, an integer $1 \leq k < d$ and $\theta \in (0,1)$: is it true that  $\as P_V(X)$ is the same for almost all $V \in G(d,k)$? 
\end{ques}

An intriguing special case of this question concerns the quasi-Assouad dimension (that is, when $\theta\to1$).   On the other hand, for $\theta\to 0$ we know the answer to the question  is yes and that the almost sure constant value is given by the upper box dimension profile $\ubd^k  X$. 

We do know that the intermediate dimensions satisfy a Marstrand theorem \cite{bff21}.  Here the almost sure constant is given by a continuum of dimension profiles which interpolate between $\min\{k, \hd X\}$ and $\ubd^k X$ (continuously at $\theta = 1$ but not necessarily at $\theta=0$). We will discuss this further in Section \ref{intsec}; in particular, see Theorem \ref{main}.  However, the question is open for the Fourier spectrum. The following  question was raised in \cite{ana}.

\begin{ques}\label{ques2}
  Is there a Marstrand theorem for the Fourier spectrum?  More precisely, fix a compact set  or finite compactly supported  Borel measure in $\rd$, an integer $1 \leq k < d$ and $\theta \in [0,1]$: is it true that  the value of the Fourier spectrum at $\theta$ of the projection onto $V$ is the same for almost all $V \in G(d,k)$? 
\end{ques}

An intriguing special case of this question concerns the Fourier dimension (that is, when $\theta=0$) but we are unaware of any progress on this front.  This question was discussed during the problem sessions of the Banff conference in 2024.  On the other hand, for $\theta=1$ we know the answer to the question  is yes for sets  by the Marstrand--Mattila theorem   and for measures $\mu\in\M(\rd)$ with $\sd\mu\leq k$ by \cite[Theorem~1.1]{HK97}.

We note that if the `pointwise' Questions \ref{ques1} and \ref{ques2}  (that is, for $\theta$ fixed) could be answered in the affirmative then, using continuity of the Fourier and Assouad spectra, the results could be upgraded to hold almost surely for all $\theta \in (0,1)$ simultaneously.

Finally, we briefly elaborate on the Assouad dimension situation.  Despite there not being a Marstrand theorem for Assouad dimension, it is true that for almost all $V \in G(d,k)$
\begin{equation} \label{assthm}
\ad P_V(X) \geq \min\{\ad X , k\},
\end{equation}
that is, one gets the expected lower bound almost surely.  This was proved in the case $d=2$ in \cite{FO17} and in general in \cite{fraserisrael}.  However, Orponen \cite{orponenassouad} proved something far stronger in the planar case $d=2$.  It turns out that the set of exceptions to \eqref{assthm} is not only of measure zero but also has Hausdorff dimension zero.  This beautiful result has had numerous applications in diverse directions.  For instance it was used to prove the Assouad dimension analogue of the Falconer distance problem in the plane \cite{fraserdistance}, it was used to study embeddings of self-similar sets \cite{hochman} (a problem which \emph{a priori} has nothing to do with Assouad dimension),  and to study the conformal Assouad dimension of general classes of self-affine sets in \cite{rutaraffine}.

\section{Using the Assouad spectrum to study   box dimension profiles} \label{asssec}

In this section we will prove the following theorem, which is from \cite{ffs}.  It shows that the Assouad spectrum can be used to bound the almost sure value of the upper box dimension of orthogonal projections from below. Said differently, the Assouad spectrum can be used to give a non-trivial lower bound for the box dimension profiles.

\begin{thm}\label{boxapp0}
Let $1\leq k < d$ and  $\theta \in (0,1)$. If  $X\subseteq \rn$ is bounded then, for almost all $V \in G(d,k)$,
\be\label{ubtheta}
\ubd P_V(X) \geq \ubd X - \max \{0, \ \overline{\dim}_\textup{A}^\theta X-k, \ (\ad X-k)(1-\theta)\}.
\ee
\end{thm}

In fact, the same conclusion holds with $\ubd$ replaced by $\lbd$, see \cite{ffs}.  In the absence of a precise result, a natural question is when \eqref{ubtheta} improves on the general lower bounds from \eqref{boxdimsbounds}.  A careful analysis of the lower bound yields many such situations.  We state  one   here and leave others to the reader.

\begin{cor}\label{boxapp3}
Let $1\leq k<d$ be integers. If $X\subseteq \rn$ is bounded, then for almost all $V \in G(d,k)$,
\[
\ubd P_V(X) \geq \ubd X - \max\{0, \qad X - k\}.
\]
In particular, if $\qad X \leq \min\{k, \ubd X\}$, then for almost all $V \in G(d,k)$,
\[
\ubd P_V (X) = \min\{k, \ubd X\}.
\]
\end{cor}

\begin{proof}
The lower bound follows from Theorem \ref{boxapp0} by letting $\theta \to 0$ in \eqref{ubtheta}.  The equality in the second case comes from the established lower bound and the upper bound from  \eqref{boxdimsbounds}.
\end{proof}

The following theorem relates the dimension profiles to the Assouad spectrum. Combined with Theorem \ref{mainA} this proves Theorem \ref{boxapp0}. Indeed, Theorem \ref{boxapp0} is immediate on substituting the inequalities of  Theorem \ref{boxapp} with $s=k$ in Theorem \ref{mainA}(i). It can also be used in conjunction with Theorem \ref{mainA} to establish results about the exceptional set in the box dimension analogue of Marstrand's theorem, but we refer the reader to \cite{ffs} for more on this aspect.

\begin{thm}\label{boxapp}
Let    $s\in (0,d]$  and $\theta \in (0,1)$. If $X\subseteq \rn$ is bounded then
\be\label{ineqprof}
\ubd^s X \geq \ubd  X - \max \{0, \ \overline{\dim}_\textup{A}^\theta X-s, \ (\ad X-s)(1-\theta)\}.
\ee
\end{thm}

\begin{proof}
We may assume  without loss of generality that $|X|<1/2$.  Throughout this proof we write $N_r(E)$ to denote the maximal size of an $r$-separated subset of $E$. Fix   $\alpha> \overline{\dim}_\textup{A}^\theta X$, $\beta> \ad X$ and let  $C>0$ be a constant such that for all $0<r<R<1$ and $x \in X$
\[
N_r(B(x,R) \cap X) \leq C \left(\frac{R}{r}\right)^\beta,
\]
and  for all $0<r \leq R^{1/\theta}<R<1$ and $x \in X$
\[
N_r(B(x,R) \cap X) \leq C \left(\frac{R}{r}\right)^\alpha.
\]
Let  $0< r<1$ and  $\{x_i\}_{i=1}^{N_r(X)}$ be a maximal $r$-separated set of points in $X$.  Place a point mass of weight $1/N_r(X)$ at  each $x_i$ and let the measure $\mu$ be the sum of these point masses, noting that $\mu(X) = 1$.

Write $D= \lceil \log_2(2|X|r^{-1})\rceil$ and $B=\lceil (1-\theta)\log_2(r^{-1})\rceil$ noting that for sufficiently small $r$, $1\leq B<D$.     Recalling \eqref{ker}, for each $i$ the potential of $\mu$ at $x_i$ is
\begin{eqnarray*}
\int  \phi_r^s(x_i-y)d\mu(y)  &\leq& \sum_{n=0}^D 2^{-(n-1)s} \mu(B(x_i,2^n r)) \\ \\
&\leq& \sum_{n=0}^D 2^{-(n-1)s}\frac{1}{N_r(X)} N_r\big(B(x_i,2^n r)\cap X\big) \\ \\
&\leq&  \frac{2^s}{N_r(X)} \bigg( \sum_{n=B}^D 2^{-ns}C\Big(\frac{2^n r}{r}\Big)^\alpha \ + \ \sum_{n=0}^{B-1} 2^{-ns}C\Big(\frac{2^n r}{r}\Big)^\beta\bigg) \\ \\
&\leq& c\, \frac{\max\{1, \ r^{-(\alpha-s)}, \ r^{-(\beta -s)(1-\theta)}\}}{N_r(X) }
 \end{eqnarray*}
for a constant $c$ which is independent of $r$. Summing over the $x_i$, the energy of $\mu$ is
$$\int \int \phi_r^s(x-y)d\mu(x)d\mu(y) \leq c\frac{\max\{1, \ r^{-(\alpha-s)}, \ r^{-(\beta-s)(1-\theta)}\}}{N_r(X) }$$
and so the capacity $C_r^s(X)$ satisfies
\[
C_r^s(X)\geq c^{-1}N_r(X) \min\{1, \ r^{(\alpha-s)}, \ r^{(\beta-s)(1-\theta)}\} .
\]
Thus
\begin{eqnarray*}
 \overline{\mbox{\rm dim}}_{\rm B}^s X  &=&   \varlimsup_{r\to 0} \frac{\log  C_r^s(X)}{-\log r}\\ \\
& \geq& \varlimsup_{r\to 0} \frac{\log \left(c^{-1}N_r(X) \min\{1, \ r^{(\alpha-s)}, \ r^{(\beta-s)(1-\theta)}\} \right)}{-\log r}\\ \\
& =& \ubd  X - \max \{0, \ \alpha-s, \ (\beta-s)(1-\theta)\}.
\end{eqnarray*}
The conclusion  follows on taking $\alpha$ and $\beta$ arbitrarily close to $\overline{\dim}_\textup{A}^\theta F$ and $\ad F$ respectively.
\end{proof}

\section{Using the intermediate dimensions to study box  dimension profiles} \label{intsec}

In this section we exhibit work from \cite{bff21}.  In particular, we introduce a capacity approach to studying intermediate dimensions and use this to state a Marstrand theorem for intermediate dimensions; see Theorem \ref{main}.  We use this to establish an intriguing result which relates the box dimensions of projections to continuity of the intermediate dimensions at 0; see Corollary \ref{cor4}.

Throughout this section, let $\theta \in (0, 1]$ and $k \in \{1, \dots, d\}$. For $0 \leq s \leq k$ and $0<r <1$, define   kernels
\be\label{ker}
{\phi}_{r, \theta}^{s, k}(x) = \begin{cases} 
      1 & 0\leq |x| < r \\
      \big(\frac{r}{|x|}\big)^s & r\leq |x| < r^\theta   \\
      \frac{r^{\theta(k-s) + s}}{|x|^k}\ & r^\theta \leq |x|
   \end{cases} \qquad (x\in \R^d).
\ee
When $s = k$ these  correspond  to the kernel $\phi_r^k(x)$ used in \cite{falconerprofile} in the context of box  dimensions; see \eqref{ker}.    Letting $\mathcal{M}(X)$ denote the set of Borel probability measures supported on $X$, we say that the \emph{energy} of $\mu \in \mathcal{M}(X)$ with respect to $\phi_{r, \theta}^{s, k}$ is  
\begin{equation*}
\int\int \phi_{r, \theta}^{s, k}(x - y) \,d\mu(x)d\mu(y)
\end{equation*}
and the {\em potential} of $\mu$ at $x \in \R^d$ is
\begin{equation*}
\int \phi_{r, \theta}^{s, k}(x - y)\,d\mu(y).
\end{equation*}
We define the \emph{capacity} $C_{r, \theta}^{s, k}(X)$ of a compact set $X$ to be the reciprocal of the minimum energy achieved by probability measures on $X$, that is
\begin{equation*}
{C_{r, \theta}^{s, k}(X)} = \left(\inf\limits_{\mu \in \mathcal{M}(X)} \int\int \phi_{r, \theta}^{s, k}(x - y) \,d\mu(x)d\mu(y)\right)^{-1}.
\end{equation*}
Since $\phi_{r, \theta}^{s, k}(x)$ is continuous in $x$ and strictly positive and $X$ is compact, $C_{r,\theta}^{s, k}(X)$ is positive and finite. For bounded, but not necessarily closed, sets we take the capacity  of a set to be that of its closure.   Once again this should be compared with the potential theoretic approach to Hausdorff dimension described in Section \ref{intro}  and also  the potential theoretic approach to box dimension described in Section \ref{boxpotential}. 
 
For each integer $1 \leq k \leq d$, we define the \emph{lower intermediate dimension profile} $\lid^k X$  of $X \subset \R^d$ to be   the unique  $s\in [0,k] $  such that  
\[
 \lowlim\limits_{r \rightarrow 0}\frac{\log C_{r, \theta}^{s, k}(X)}{-\log r} = s
\]
and the \emph{upper intermediate dimension profile} $\uid^k X$ to be the unique $s\in [0,k] $  such that  
\[
\uplim\limits_{r \rightarrow 0}\frac{\log C_{r, \theta}^{s, k}(X)}{-\log r} = s.\]
The fact that these notions are well-defined is established precisely in \cite{bff21} using appropriate variants on standard potential theoretic arguments.  Further, observe that the intermediate dimension profiles are increasing in $k$. This follows immediately noting that the kernels $ \phi_{r, \theta}^{t, k}(x)$ are   \emph{decreasing} in $k$.

We can now  characterise intermediate dimensions of sets $X\subset \R^d$  in terms of  the capacities $C_{r, \theta}^{s, d}(X)$ via the intermediate dimension profiles.

\begin{thm}\label{equivdim}
Let $X \subset \R^d$ be bounded and $\theta \in (0, 1]$. Then
\begin{equation*}
\lid X = \lid^{d} X
\end{equation*}
and
\begin{equation*}
\uid X = \uid^{d} X.
\end{equation*}
\end{thm}

This characterisation was proved in \cite{bff21}.  Below we extract a small piece of that proof which shows how to relate capacities to the covers used in the definition of intermediate dimensions.  This is just supposed to give the reader a flavour of the argument and we refer to \cite{bff21} for the full details.   We do not lose any generality by considering only compact sets since the intermediate dimensions are stable under taking closure for $\theta \in (0,1]$.

\begin{lma}\label{approxeq}
Let $X \subset \R^d$ be compact, $\theta \in (0, 1]$, and $0\leq s \leq d$. Then  for all $0<r\leq 1$ and all finite covers  $\{U_i\}_{i }$ of $X$ by sets of diameters $r \leq |U_i| \leq r^\theta$, 
\[
\sum\limits_{i } |U_i|^s \geq r^sC_{r,\theta}^{s, d}(X).
\]
\end{lma}

\begin{proof}
 By standard    potential theoretic  arguments (see e.g.~\cite[Lemma 2.1]{falconerprofile}) there exists an equilibrium measure $\mu \in \mathcal{M}(X)$ and a set $X_0$ with $\mu(X_0) = 1$ such that
$$\int \phi_{r, \theta}^{s, d}(x - y) d\mu(y) = \frac{1}{C_{r, \theta}^{s, d}(X)} =: \gamma$$
for all $x \in X_0$. Let $r \leq \delta \leq r^\theta$ and $x \in X_0$. Then
\begin{equation}\label{zz}
\gamma = \int {\phi}_{r, \theta}^{s, d}(x - y) d\mu(y) 
\geq \int \left(\frac{r}{\delta}\right)^s1_{B(0, \delta)}(x - y) d\mu(y) 
\geq \left(\frac{r}{\delta}\right)^s \mu(B(x, \delta)).
\end{equation}
Let $\{U_i\}_{i }$ be a finite cover of $X$ by sets of diameters $r \leq |U_i| \leq r^\theta$ and define $\mathcal{I} = \{i : U_i \cap X_0 \neq \emptyset \}$. Then for each $i \in \mathcal{I}$, there exists $x_i \in U_i \cap X_0$ so that $U_i \subset B(x_i, |U_i|)$. Hence
\begin{equation*}
1 = \mu(X_0) \leq \sum\limits_{i \in \mathcal{I}} \mu(U_i) 
\leq  \sum\limits_{i \in \mathcal{I}} \mu(B(x_i, |U_i|))\leq r^{-s}\gamma \sum\limits_{i \in \mathcal{I}}|U_i|^s
\end{equation*}
by (\ref{zz}), and so
\begin{equation*}
\sum\limits_{i } |U_i|^s \geq r^sC_{r,\theta}^{s, d}(X)
\end{equation*}
as required.
\end{proof}

The main theorem proved in \cite{bff21} is the following.  We almost have the ingredients in place to prove it, but we omit the proof and refer the reader to \cite{bff21} for the details.  It establishes a Marstrand theorem for the intermediate dimensions.

\begin{thm}\label{main}
Let $X \subset \R^d$ be bounded. Then, for  all $V \in G(d, k)$
\begin{equation}\label{mains}
\lid   P_V (X) \leq \lid^{k} X \quad\mathrm{and}\quad \uid   P_V (X) \leq \uid^{k} X
\end{equation}
for all    $\theta\in (0,1]$.  Moreover, for $\gamma_{d, k}$-almost all $V \in G(d, k)$,
\begin{equation}\label{mainas}
\lid   P_V (X) = \lid^{k} X \quad\mathrm{and}\quad \uid   P_V (X) = \uid^{k} X
\end{equation}
for  all $\theta\in (0,1]$.
\end{thm}

The next result allows us to pass continuity of the intermediate dimensions at zero to typical projections.  This will be useful in subsequent applications.

\begin{cor} \label{cor1}
Let $X \subset \mathbb{R}^d$ be a bounded set such that $\lid X$ is continuous at $\theta=0$.  If $V \in G(d, k)$ is such that $\hd P_V (X) \geq \min\{k, \hd X\},$ then $\lid  P_V (X)$ is continuous at  $\theta=0$.  In particular, $\lid  P_V (X)$ is continuous at $\theta =0$ for    $\gamma_{d, k}$-almost all $V \in G(d, k)$. A similar result holds for the upper intermediate dimensions.
\end{cor}

\begin{proof}
If $k\leq  \hd X$, then the result is immediate and so we may assume that $k>  \hd X$.  Then, for $\theta \in (0,1)$, using  Theorems \ref{equivdim}, \ref{main} and the assumption that $\lid X$ is continuous at $\theta=0$, we get
\[
\hd X \leq \hd  P_V (X) \leq  \lid  P_V (X) \leq   \lid^k X \leq   \lid^d  X =  \lid   X \to \hd X
\]
as $\theta \to 0$, which proves continuity of  $\lid  P_V (X) $ at $\theta =0$.   The final part of the result, concerning almost sure continuity at 0,  follows from the previous observation together with the Marstrand--Mattila projection theorem for Hausdorff dimensions.
\end{proof}

Next we recall the useful observation that if the box dimension of a set equals the ambient spatial dimension, then the intermediate dimensions must be constantly equal to the box dimension for all $\theta \in (0,1]$, see \cite[Proposition 2.4]{intdims}.   The  capacity approach presented here yields a simple proof, which we give now.
\begin{cor} \label{cor3}
If $X \subset \R^d$ is bounded and satisfies $\lbd X = d$, then  $
\lid X = \uid X =  d$ for all $\theta \in (0,1]$.  Similarly,  if $\ubd X = d$, then  $
 \uid X =  d$ for all $\theta \in (0,1]$.
\end{cor}
\begin{proof}
Observe that
$$
\lowlim\limits_{r \rightarrow 0}\frac{\log C_{r, \theta}^{d, d}(X)}{-\log r} =     \lbd X = d
$$
and so by   Theorem \ref{equivdim} it follows $  \uid X \geq \lid  X =  \lid^d X = \lbd X = d$.  The result concerning $\uid X$ alone follows similarly.
\end{proof}

The following counter-intuitive result follows by piecing together Corollaries \ref{cor1} and \ref{cor3}.  This gives a concrete application of the intermediate dimensions to a question concerning only the box and Hausdorff dimensions.

\begin{cor}\label{cor4}
Let $X \subset \mathbb{R}^d$ be a bounded set such that $\lid X$ is continuous at $\theta=0$.  Then
\[
\lbd P_V(X) = k
\]
for $\gamma_{d, k}$-almost all $V \in G(d, k)$ if and only if
\[
\hd X \geq k.
\]
A similar result holds for upper dimensions replacing $\lid X$ and $\lbd X$ with $\uid X$ and $\ubd X$, respectively.
\end{cor}
\begin{proof}
One direction is trivial, and holds without the continuity assumption, since, if $\hd X \geq k$, then 
\[
k \geq \lbd P_V(X)  \geq \hd P_V(X)  \geq k
\]
for $\gamma_{d, k}$-almost all $V \in G(d, k)$.  The other direction is where the interest lies.  Indeed, suppose  $\lbd P_V(X) = k$ for $\gamma_{d, k}$-almost all $V \in G(d, k)$ but 
$\hd X < k$.  Then Corollary \ref{cor3} implies that $\lid  P_V(X) = k$ for $\gamma_{d, k}$-almost all $V \in G(d, k)$ and all $\theta \in (0,1]$.  Applying the Marstrand--Mattila projection theorem for Hausdorff dimension, it follows that for  $\gamma_{d, k}$-almost all $V \in G(d, k)$    $\lid  P_V(X) $ is not continuous at $\theta = 0$, which contradicts Corollary \ref{cor1}.
\end{proof}

To motivate Corollary \ref{cor4} we give a couple of simple applications, both taken from \cite{bff21}.  First, if $X \subset \mathbb{R}^2$ is a self-affine  Bedford--McMullen carpet satisfying $\hd X < 1 \leq \bd X$, then, since $\dim_\theta X$ is continuous at 0, 
\[
\ubd P_V(X)<1 = \min \{ \bd X , 1\}
\]
for    $\gamma_{2, 1}$-almost all $V \in G(2, 1)$.  This surprising application seems difficult to derive directly.  The fact that the intermediate dimensions of Bedford--McMullen carpets are continuous at 0 was first established in \cite{intdims}; see also \cite{banajicarpets}.

Another, more accessible, example is provided by the sequence sets $F_p = \{ n^{-p} : n \geq 1\}$ for fixed $p >0$.  It is well-known that $\bd F_p = 1/(1+p)$ and therefore
\[
\bd (F_p \times F_p) = 2/(1+p)
\]
which is at least 1 for $p \leq 1$ and approaches 2 as $p$ approaches 0.   Continuity at $\theta = 0$ for $\uid F_p$ was established in \cite[Proposition 3.1]{intdims} and it is straightforward  to extend this to  $\uid F_p \times F_p$.  Therefore, since $\hd  F_p \times F_p=0<1$, we get
\[
\ubd P_V (F_p \times F_p) < 1
\]
for  $\gamma_{2, 1}$-almost all $V \in G(2, 1)$.  This is most striking when $p$ is very close to 0.  The reader may enjoy trying to derive the following:   for all $V \in G(2, 1)$ apart from the horizontal and vertical projections
\[
\ubd P_V( F_p \times F_p)  =  1- \left(\frac{p}{p+1}\right)^2.
\]
We would never have come across this entertaining formula   if Corollary \ref{cor4} had not led us to it.

\section{Using the Fourier spectrum to study the exceptional set for Hausdorff dimension} \label{specsec}

In this section we exhibit work from \cite{ana} which used the Fourier spectrum to prove new estimates for the Hausdorff dimension of the exceptional set in the Marstrand--Mattila projection theorem.  Specifically, we will prove the following result and describe a concrete application concerning continuity of information coming from the Fourier dimension, see Corollary \ref{prop:continuity}.  Here and throughout we write $\mu_V$ for the pushforward of $\mu$ under $P_V$, that is, $\mu_V$ is the projection of $\mu$ onto $V$.

\begin{thm}\label{thm:exceptionalFouriercoro}
	Let $\mu\in\M(\rd)$ and $1\leq k<d$ be an integer. Then for all $u\in[0,k]$,
	\begin{equation*}
		\hd \{ V\in\gdk : \hd \mu_{V}<u \}\leq\max\bigg\{0, k(d-k)+\inf_{\theta\in(0,1]}\frac{u-\fs  \mu}{\theta} \bigg\}.
	\end{equation*}
	Furthermore, if $X$ is a Borel set in $\rd$, then for all $u\in[0,k]$,
	\begin{equation*}
		\hd \{ V\in \gdk : \hd  P_{V}(X)<u \}\leq \max\bigg\{0, k(d-k)+\inf_{\theta\in(0,1]}\frac{u-\fs  X}{\theta} \bigg\}.
	\end{equation*}
\end{thm}

  Since for all $\theta\in[0,1]$, $\min\{ d, \fs \mu \}\leq \hd \mu$ and $\fs X\leq \hd X$, the previous theorem is an  immediate consequence of Theorem~\ref{thm:exceptionalFourier}, which we state and prove later.

Before bringing the Fourier spectrum into play, we first consider what information one can get concerning the exceptional set by considering the Fourier dimension alone.  Indeed, it is easy to see that there are \emph{no} exceptional projections for Salem sets, that is,  sets that have the same Fourier and Hausdorff dimension.  In fact, one can say   more.  If $X\subseteq\rd$ is a Borel set with $\fd X = t$, then for all $u\leq t$,
\begin{equation*}
		\big\{ V\in\gdk : \hd P_{V}(X) < \min\{k,u\} \big\} =\varnothing,
\end{equation*}
see \cite{ana} for more details.  A natural question arising from this is whether Fourier dimension can be used to say anything   for $u>t$; for example, is there a continuous upper bound for 
\begin{equation*}
	\hd 	\big\{ V\in\gdk : \hd P_{V}(X) < \min\{k,u\} \big\}  
\end{equation*}
as a function of $u$ (depending on the Fourier dimension) which is 0 at $u=t$? Perhaps surprisingly, it turns out that  nothing can be said  using the Fourier dimension alone.  That is, knowledge of the Fourier dimension of $X$ does not give any information about the dimension of the set of $V$ for which $\hd P_{V}(X) < u$ as soon as $u> \fd X$.  The following was proved in \cite{ana}.

\begin{prop}\label{propo:counter_example}
	For any $s\in(0,1]$ and $t\in(s/2,s)$ there exists a compact set $X\subseteq\R^2$ with $\hd X = s$ and $\fd X = t$ such that for $u<t$,
\begin{equation*}
    \{ V\in\gto : \hd P_{V}(X)\leq u \} =\varnothing,
\end{equation*}
and for $u\geq t$,
\begin{equation*}
  \hd\{ V\in\gto : \hd P_{V}(X)\leq u \} \geq 2t-s.
\end{equation*}
That is, the dimension of the exceptional set has a jump discontinuity at $\fd X$ from $0$ to the largest value it could possibly take according to \eqref{eq:oberlin}.
\end{prop}

We will see that one can say more by using the Fourier spectrum.  We now state and prove the main result from \cite{ana}.  Observe that when  $\theta = 1$ we recover the bound from \cite[Proposition~6.1]{PS00}.

\begin{thm}\label{thm:exceptionalFourier}
	Let $\mu\in\M(\rd)$, $\theta\in(0,1]$ and $1\leq k<d$ be an integer. Then for all $0\leq u\leq \min\{ k, \fs \mu \}$,
 \begin{equation}\label{eq:thm1}
 	 \hd \{ V\in G(d,k) : \fs \mu_{V}<u \}\leq  \max\bigg\{ 0, k(d-k)+\frac{u-\fs \mu}{\theta}\bigg\}.
 \end{equation}
 Furthermore, if $X$ is a Borel set in $\rd$ and $\theta\in(0,1]$, then for all $0\leq u\leq \min\{ k, \fs X \}$,
 \begin{equation*}
 	\hd \{ V\in G(d,k) : \fs  P_{V}(X)<u \}\leq  \max\bigg\{ 0, k(d-k)+\frac{u-\fs X}{\theta}\bigg\}.
 \end{equation*}
\end{thm}
\begin{proof}
In this proof we write $A\lesssim B$ if there exists a constant $C>0$ such that $A\leq CB$ and $A\approx B$ if $A\lesssim B$ and $B\lesssim A$. If we wish to emphasise that the constant $C$ depends on some parameter $\lambda$ we shall express it as $A\lesssim_{\lambda}B$ or $A\approx_{\lambda}B$.

    The claim for sets follows from the statement for measures. To see this fix $\theta\in(0,1]$, let $\varepsilon>0$ and $\mu\in\M(X)$ be such that $\fs\mu \geq  \fs X - \varepsilon$. Since $\fs P_{V}(X) \geq \min\{ k, \fs \mu_{V}\}$ and $u \leq k$, it follows that $\fs P_{V}(X)<u \Rightarrow  \fs\mu_{V}<u$.  Then,
  \begin{align*}\label{eq:suff}
    \hd\{ V\in\gdk : \fs P_{V}(X)<u \}&\leq\hd\{ V\in\gdk : \fs\mu_{V}<u \}\\
    &\leq k(d-k)+\frac{u-\fs\mu}{\theta}\\
    &= k(d-k)+\frac{u-\fs X+\varepsilon}{\theta},
  \end{align*}
 and  letting $\varepsilon\to0$ gives the result. We now proceed to prove the claim for measures, which follows \cite{ana}.
   
  Let $G_{u,\theta} = \{ V\in\gdk : \fs  \mu_{V}<u \}$ and suppose \eqref{eq:thm1} is false for some $u>0$. Choose $\tau>0$ such that $k(d-k)+\frac{u-s}{\theta}<\tau<\hd G_{u,\theta}$, for some $s<\fs \mu$.  First, observe that $G_{u,\theta}$ is a Borel set, see \cite{ana} for details.  Therefore, by Frostman's lemma there exists a measure $\nu\in\M(G_{u,\theta})$ such that $\nu\big( B(V,r) \big)\leq r^\tau$ for all $V\in\gdk$ and $r>0$. We will show that
	\begin{equation}\label{eq:thm1_energy}
		\int_{\gdk}\J_{u,\theta}(\mu_{V})^{1/\theta}\,d\nu(V)< \infty,
	\end{equation}
	and this will imply that $\J_{u,\theta}(\mu_{V})^{1/\theta}<\infty$ for $\nu$ almost every $V\in \gdk$. Then $\nu(G_{u,\theta})=0$ which contradicts the fact that $\nu\in\M(G_{u,\theta})$.
	
Given $y \in \mathbb{R}^k$ and $V\in\gdk$, define  $y_{V} \in \rd$ by
\begin{equation}\label{eq:yv}
\{y_{V} \} = V\cap P_{V}^{-1}(y).
\end{equation}
We write $\S(\rd)$ for the family of  functions in the Schwartz class on $\rd$; see  \cite[Chapter 3]{Mat15}.	Let $\varphi\in\S(\rd)$ be such that $\varphi(x) = 1$ in $\spt\mu$, where $\spt\mu$ denotes the (compact) support of $\mu$.  Then $\mu = \varphi \mu$ and $\widehat{\mu} = \widehat{\varphi \mu} = \widehat{\varphi}*\widehat{\mu}$. Moreover,  $\widehat{\varphi}\in\S(\rd)$, and for every $N\in\N$, $\big|\widehat{\varphi}(z)\big|\lesssim_{\varphi,N}\big(1+|z|\big)^{-N}$. Using H\"older's inequality with conjugate exponents $2/\theta$ and $2/(2-\theta)$, we obtain the following estimate for $z \in \rd$:
	\begin{align*}
		\big|\widehat{\mu}(z)\big|^{\frac{2}{\theta}} &\leq \bigg[ \int_{\rd}\big|\widehat{\mu}(z-y)\widehat{\varphi}(y)\big|\,dy \bigg]^{\frac{2}{\theta}}\\
		&= \bigg[ \int_{\rd} \big|\widehat{\mu}(z-y)\big|\big|\widehat{\varphi}(y)\big|^{\frac{\theta}{2}} \big|\widehat{\varphi}(y)\big|^{\frac{2-\theta}{2}}\,dy\bigg]^{\frac{2}{\theta}}\\
		&\leq \Bigg[ \bigg( \int_{\rd}\big| \widehat{\mu}(z-y) \big|^{\frac{2}{\theta}} \big| \widehat{\varphi}(y) \big|\,dy\bigg)^{\frac{\theta}{2}}\bigg( \int_{\rd} \big| \widehat{\varphi}(y) \big|\,dy \bigg)^{\frac{2-\theta}{2}} \Bigg]^{\frac{2}{\theta}}\\
		&= \int_{\rd} \big| \widehat{\mu}(z-y) \big|^{\frac{2}{\theta}}\big| \widehat{\varphi}(y) \big|\,dy \,\bigg(\int_{\rd} \big| \widehat{\varphi}(y) \big|\,dy \bigg)^{\frac{2-\theta}{\theta}}\\
		&\lesssim \int_{\rd}\big| \widehat{\mu}(z-y) \big|^{\frac{2}{\theta}}\big| \widehat{\varphi}(y) \big|\,dy \\
&= \big( \big|\widehat{\varphi}\big| * \big|\widehat{\mu}\big|^{\frac{2}{\theta}} \big)(z).
	\end{align*}	
	 Recalling \eqref{eq:yv}, $\widehat{\mu_{V}}(y) = \widehat{\mu}(y_{V})$, and therefore
	\begin{align*}
		\int_{\gdk}\J_{u,\theta}(\mu_{V})^{1/\theta}\,d\nu(V) &= \int_{\gdk}\int_{\rk} \big| \widehat{\mu_{V}}(y) \big|^{\frac{2}{\theta}}|y|^{\frac{u}{\theta}-k} \, dy \,d\nu(V)\\
		&= \int_{\gdk}\int_{\rk}\big| \widehat{\mu}(y_{V}) \big|^{\frac{2}{\theta}}|y|^{\frac{u}{\theta}-k}\, dy \,d \nu(V)\\
		&\lesssim \int_{\gdk}\int_{\rk}\big( \big|\widehat{\varphi}\big|*\big|\widehat{\mu}\big|^{\frac{2}{\theta}} \big)(y_{V})|y|^{\frac{u}{\theta}-k}\, dy \,d\nu(V)\\
		&= \int_{\gdk}\int_{\rk}\bigg( \int_{\rd}\big| \widehat{\varphi}(y_{V}-z) \big| \big| \widehat{\mu}(z) \big|^{\frac{2}{\theta}}\,dz\bigg)|y|^{\frac{u}{\theta}-k}\, dy \,d\nu(V)\\
		&= \int_{\rd}\big| \widehat{\mu}(z) \big|^{\frac{2}{\theta}}\Bigg[\int_{\gdk}\int_{\rk} \big| \widehat{\varphi}(y_{V}-z) \big||y|^{\frac{u}{\theta}-k} \, dy \,d\nu(V)\Bigg]\,dz\\
		&\lesssim \int_{\rd}\big| \widehat{\mu}(z) \big|^{\frac{2}{\theta}}\Bigg[ \int_{\gdk}\int_{\rk}\big(1+|y_{V}-z|\big)^{-N}|y|^{\frac{u}{\theta}-k}\, dy \,d\nu(V) \Bigg]\,dz.
	\end{align*}
	
	To finish the proof of the theorem we need to show that the above is finite. It is enough to see that for $N$ sufficiently large,
	\begin{equation}\label{eq:thm_bound}
		\int_{\gdk}\int_{\rk}\big(1+|y_{V}-z|\big)^{-N}|y|^{\frac{u}{\theta}-k}\, dy \,d\nu(V)\lesssim |z|^{\frac{u}{\theta}+k(d-k)-d-\tau},
	\end{equation}
for all $z \in \rd$ with $|z| \geq 2$ because then, since $k(d-k)+\frac{u-s}{\theta}<\tau$,
	\begin{align*}
		\int_{\gdk}\J_{u,\theta}(\mu_{V})^{1/\theta}\,d\nu(V) 
		&\lesssim \int_{\rd} \big| \widehat{\mu}(z) \big|^{\frac{2}{\theta}}|z|^{\frac{u}{\theta}+k(d-k)-d-\tau}\,dz \\
		&\lesssim \int_{\rd}\big| \widehat{\mu}(z) \big|^{\frac{2}{\theta}}|z|^{\frac{s}{\theta}-d}\,dz\\
		&= \J_{s,\theta}(\mu)^{1/\theta}<\infty,
	\end{align*}
	since $s<\fs \mu$.  This establishes \eqref{eq:thm1_energy} and completes the proof.

The proof of \eqref{eq:thm_bound} is technical and we refer the reader to \cite{ana} for the details.  Briefly, one splits the integral into the dyadic annuli centred at $z$ as
	\begin{align*}
		\int_{\gdk}\int_{\rk}\big(1+|y_{V} - z|\big)^{-N}&|y|^{\frac{u}{\theta}-k}\, dy \,d\nu(V) \\
    &\quad =\iint_{\{(V,y) : |y_{V}-z|\leq 1/2 \}} + \sum_{\{ j \geq 0 : |z|> 2^{j+1} \}}\iint_{\{ (V,y) : 2^{j-1}<|y_{V}-z|\leq 2^{j} \}} \\
    &\quad \quad + \sum_{\{ j \geq 0 : |z|\leq 2^{j+1} \}}\iint_{\{ (V,y) : 2^{j-1}<|y_{V}-z|\leq 2^j \}}
	\end{align*}
where the sums are over integer $j$.  Then each term  is handled separately, using that    from the definition of $\nu$ we have for all $r>0$ and $z\in \rd$,
	\begin{equation} \label{frostmancondition}
		\nu\big( \{ V\in\gdk : d(z,V)\leq r \} \big)\lesssim \bigg( \frac{r}{|z|}\bigg)^{\tau-(k-1)(d-k)}
	\end{equation}
where    $d(z,Y) = \inf\{ |z-y| :  y \in Y\}$.  	
\end{proof}

\subsection{Continuity of the dimension of the exceptional set}

Proposition~\ref{propo:counter_example} showed us that, for sets $X$, the dimension of the set of  exceptional directions can be discontinuous at $u = \fd X$.  Earlier we hinted at the following question:  under which conditions on $\mu$ can  continuity of the dimension of the set of  exceptional directions   be recovered at  $u = \fd X$.   We show in the following corollary  that such   conditions can be given in terms of the derivative of the Fourier spectrum at 0.  There is an analogous result for measures, which we leave to the reader to formulate.  

\begin{cor}\label{prop:continuity}
Let $X$ be a Borel set in $\rd$ and let
\[
D = \liminf_{\theta \to 0} \frac{\fs X - \fd X}{\theta}
\]
be the lower right semi-derivative of $\fs X$ at $\theta = 0$.  If $D \geq k(d-k)$, then the function  $u\mapsto\hd \{ V\in\gdk : \hd P_{V}(X)<u \}$ is continuous at $u = \fd X$.
\end{cor}
\begin{proof}
Let $\varepsilon\in(0,1)$  and consider $u = \fd X + \varepsilon^2$.  Corollary \ref{thm:exceptionalFouriercoro} gives that  
	\begin{align*}
		\hd \big\{ V\in \gdk : \hd  P_{V}(X)<& \fd X + \varepsilon^2 \big\}\\
    &\leq \max\bigg\{0, k(d-k)+\inf_{\theta\in(0,1]}\frac{ \fd X + \varepsilon^2-\fs  X}{\theta} \bigg\}\\
&\leq \max\bigg\{0, k(d-k)+  \varepsilon-\frac{ \fd^\varepsilon X  -\fd  X}{\varepsilon} \bigg\}\\
& \to 0
	\end{align*}
as $\varepsilon \to 0$ provided $D \geq k(d-k)$, which proves the desired continuity result. 
\end{proof}

One drawback of the previous result is that the condition on the derivative of the Fourier spectrum is quite strong. Since $\fs \mu \leq \fd \mu+d\theta$ always holds,   in order to  establish continuity of the dimension of the exceptional set from  Proposition~\ref{prop:continuity}, it is necessary to have $k(d-k)\leq d$.  This is only true for the families $G(d,1)$, $G(d,d-1)$, and the special case $G(4,2)$.

\section*{Data availability statement} 

My manuscript has no associated data.

\section*{Acknowledgements} 

I have had the pleasure of discussing dimension interpolation and the dimension theory of orthogonal projections with many mathematicians over the years. There are too many to name individually, but I am especially grateful to Stuart Burrell, Kenneth Falconer, Ana de Orellana, and  Pablo Shmerkin (the co-authors with whom I wrote the papers exhibited here) and also Kenneth Falconer (again), Tom Kempton and Han Yu (the co-authors with whom I introduced the various instances of dimension interpolation discussed here). I am also grateful to the organisers and participants of the Banff meeting in June 2024.  The conference provided a stimulating atmosphere in a beautiful location.

\end{document}